\documentclass[a4paper, 11pt]{article}     

\usepackage{hyperref}

\usepackage{natbib} 
\usepackage{mathtools}
\usepackage{amsfonts,amssymb,amsmath,amsthm,latexsym,bbm,mathrsfs}
\usepackage{color}
\usepackage{geometry}  
\usepackage{textcase} 
\usepackage[utf8]{inputenc}
\usepackage[english]{babel}
\usepackage{etoolbox} 
\usepackage{tikz,pgfplots} 
\usetikzlibrary{calc}

\def\eq#1{\eqref{#1}}
\def\ben#1{\begin{equation}%
    #1\end{equation}}
\def\dsub{d_{\sub}}
\usepackage{enumitem}

\makeatletter
\numberwithin{equation}{section}
\def\cite{\citet}

\def\@noindentfalse{\global\let\if@noindent\iffalse}
\def\@noindenttrue {\global\let\if@noindent\iftrue}
\def\@aftertheorem{%
  \@noindenttrue
  \everypar{%
    \if@noindent%
      \@noindentfalse\clubpenalty\@M\setbox\z@\lastbox%
    \else%
      \clubpenalty \@clubpenalty\everypar{}%
    \fi}}

\theoremstyle{plain}
\newtheorem{theorem}{Theorem}[section]
\AfterEndEnvironment{theorem}{\@aftertheorem}
\newtheorem{definition}[theorem]{Definition}
\AfterEndEnvironment{definition}{\@aftertheorem}
\newtheorem{lemma}[theorem]{Lemma}
\AfterEndEnvironment{lemma}{\@aftertheorem}
\newtheorem{corollary}[theorem]{Corollary}
\AfterEndEnvironment{corollary}{\@aftertheorem}

\theoremstyle{definition}
\newtheorem{remark}[theorem]{Remark}
\AfterEndEnvironment{remark}{\@aftertheorem}
\newtheorem{example}{Example}[section]
\AfterEndEnvironment{example}{\@aftertheorem}
\AfterEndEnvironment{proof}{\@aftertheorem}

\def\note#1{\par\smallskip%
\noindent\kern-0.01\hsize%
\setlength\fboxrule{0pt}\fbox{\setlength\fboxrule{0.5pt}\fbox{%
\llap{$\boldsymbol\Longrightarrow$ }%
\vtop{\hsize=0.98\hsize\parindent=0cm\small\rm #1}%
\rlap{$\enskip\,\boldsymbol\Longleftarrow$}
}}%
}

\let\original@left\left
\let\original@right\right
\renewcommand{\left}{\mathopen{}\mathclose\bgroup\original@left}
\renewcommand{\right}{\aftergroup\egroup\original@right}

\def\given{\mskip 0.5mu plus 0.25mu\vert\mskip 0.5mu plus 0.15mu}
\newcounter{bracketlevel}%
\def\@bracketfactory#1#2#3#4#5#6{%
\expandafter\def\csname#1\endcsname##1{%
\global\advance\c@bracketlevel 1\relax%
\global\expandafter\let\csname @middummy\alph{bracketlevel}\endcsname\given%
\global\def\given{\mskip#5\csname#4\endcsname\vert\mskip#6}\csname#4l\endcsname#2##1\csname#4r\endcsname#3%
\global\expandafter\let\expandafter\given\csname @middummy\alph{bracketlevel}\endcsname%
\global\advance\c@bracketlevel -1\relax%
}%
}
\def\bracketfactory#1#2#3{%
\@bracketfactory{#1}{#2}{#3}{relax}{0.5mu plus 0.25mu}{0.5mu plus 0.15mu}
\@bracketfactory{b#1}{#2}{#3}{big}{1mu plus 0.25mu minus 0.25mu}{0.6mu plus 0.15mu minus 0.15mu}
\@bracketfactory{bb#1}{#2}{#3}{Big}{2.4mu plus 0.8mu minus 0.8mu}{1.8mu plus 0.6mu minus 0.6mu}
\@bracketfactory{bbb#1}{#2}{#3}{bigg}{3.2mu plus 1mu minus 1mu}{2.4mu plus 0.75mu minus 0.75mu}
\@bracketfactory{bbbb#1}{#2}{#3}{Bigg}{4mu plus 1mu minus 1mu}{3mu plus 0.75mu minus 0.75mu}
}
\bracketfactory{clc}{\lbrace}{\rbrace}
\bracketfactory{clr}{(}{)}
\bracketfactory{cls}{[}{]}
\bracketfactory{abs}{\lvert}{\rvert}
\bracketfactory{norm}{\Vert}{\Vert}
\bracketfactory{floor}{\lfloor}{\rfloor}
\bracketfactory{ceil}{\lceil}{\rceil}
\bracketfactory{angle}{\langle}{\rangle}

\newcommand{\be}[1]{\begin{equation}\label{#1}}
\newcommand{\ee}{\end{equation}}
\newcommand{\bl}[1]{\begin{lemma}\label{#1}}
\newcommand{\el}{\end{lemma}}
\newcommand{\br}[1]{\begin{remark}\label{#1}}
\newcommand{\er}{\end{remark}}
\newcommand{\bt}[1]{\begin{theorem}\label{#1}}
\newcommand{\et}{\end{theorem}}
\newcommand{\bd}[1]{\begin{definition}\label{#1}}
\newcommand{\ed}{\end{definition}}
\newcommand{\bp}[1]{\begin{proposition}\label{#1}}
\newcommand{\ep}{\end{proposition}}
\newcommand{\bc}[1]{\begin{corollary}\label{#1}}

\newcommand{\bcj}[1]{\begin{conjecture}\label{#1}}
\newcommand{\ecj}{\end{conjecture}}

\def\beq#1{\begin{equation*}#1\end{equation*}}
\def\beqn#1{\begin{equation}#1\end{equation}}
\def\beqs#1{\begin{equation*}\begin{split}#1\end{split}\end{equation*}}
\def\beqsn#1{\begin{equation}\begin{split}#1\end{split}\end{equation}}

\def\baqn#1{\begin{align}#1\end{align}}

\newcommand{\R}{\ensuremath{\mathbb{R}}}

\newcommand{\N}{\ensuremath{\mathbb{N}}}

\newcommand{\EE}{\mathop{{}\mathbb{E}}\mathopen{}}\let\IE\EE
\newcommand{\PP}{\mathop{{}\mathbb{P}}\mathopen{}}\let\IP\PP

\newcommand{\RR}{\ensuremath{\mathbb{R}}}
\newcommand{\dd}{d}

\newcommand{\cG}{\mathcal{G}} 
\newcommand{\cS}{\mathcal{S}}
\newcommand{\cP}{\mathcal{P}} 
\newcommand{\cL}{\mathcal{L}}

\newcommand{\wconv}{\Longrightarrow}
\newcommand{\asconv}{\longrightarrow}

\def\bklr#1{\bigl(#1\bigr)}
\def\klr#1{(#1)}
\def\abs#1{\vert#1\vert}
\def\babs#1{\bigl\vert#1\bigr\vert}

\def\toinf{\rightarrow \infty}
\def\cF{{\cal F}}
\def\cW{{\cal W}}
\def\eps{\varepsilon}
\def\sub{\mathrm{sub}}
\def\dsub{d_{\sub}}
\def\phi{\varphi}

\newcommand{\Def}{\coloneqq}
\newcommand{\I}{\mathop{{}\mathrm{I}}\mathopen{}}

\newcommand{\inj}{{\mathop{\mathrm{inj}}}}
\renewcommand{\hom}{{\mathop{\mathrm{hom}}}}
\newcommand{\Po}{\mathop{\mathrm{Po}}}
\def\wt{\widetilde}

\makeatother

\begin{document}    


\title{\bf Graphon-valued stochastic processes from population genetics }

\author{Siva Athreya, Frank den Hollander and Adrian R\"ollin}

\date{\itshape Indian Statistical Institute Bangalore, Leiden University \\
and National University of Singapore}


\maketitle


\begin{abstract}
The goal of this paper is to develop a theory of graphon-valued stochastic processes, and to construct and analyse a natural class of such processes arising from population genetics. We consider finite populations where individuals change type according to Wright-Fisher resampling. At any time, each pair of individuals is linked by an edge with a probability that is given by a type-connection matrix, whose entries depend on the current empirical type distribution of the entire population via a fitness function. We show that, in the large-population-size limit and with an appropriate scaling of time, the evolution of the associated adjacency matrix converges to a random process in the space of graphons, driven by the type-connection matrix and the underlying Wright-Fisher diffusion on the multi-type simplex. In the limit as the number of types tends to infinity, the limiting process is driven by the type-connection kernel and the underlying Fleming-Viot diffusion.



\end{abstract}

\makeatletter
\makeatother


\section{Introduction}
\label{intro}

In this paper we construct a class of graphon-valued Markov processes from a sequence of dynamically evolving dense graphs. First, we characterise weak convergence for processes that take values in the metric space of graphons (see Theorem~\ref{thm1}). Afterwards, using this characterisation, we find interesting scaling limits with the help of models from population genetics and use these to construct \emph{graphon-valued diffusions} (see Theorem~\ref{13} and Theorem~\ref{17}).

\subsection{Background}


Various fields of research -- including physics, computer science, sociology and epidemiology -- have produced a considerable literature on \emph{dynamics of real-world networks}. The first example of a time-changing network was proposed by \cite{Holland1977}, for the evolution of social networks. Accounts of subsequent developments are given in \cite{Snijders2001} and \cite{Snijders2010a}, including a description of statistical procedures to monitor the effects of the dynamics. In epidemiology, there is a long tradition of modelling the spread of pathogens on social contact networks. In particular, the study of sexually transmitted infections needs to take into account that partnership networks change at about the same speed as the infections spread. Most of the mathematical analysis in this area is done by means of compartmental models, which goes back to~\cite{Kermack1927}. These models lead to systems of ordinary differential equations, and consequently stochastic effects are lost. Agent-based models are also frequently used, but these are typically intractable mathematically (see, for example, \cite{Morris1997}). In physics and computer science, much of the efforts are driven by simulations or mathematical non-rigorous techniques. See \cite{Holme2012} for a survey with many real-world applications, and an attempt to unify various sub-disciplines that have emerged. 

In contrast to these efforts, the mathematical treatment of the topic is still in its early stages. Some interacting particle systems, such as oriented percolation, can be interpreted as the spread of an infection on time-varying networks, but these networks are intrinsically highly structured (typically taking the form of a lattice), and therefore are very different from the networks we have in mind in the present paper. Also, most of the work deals with \emph{sparse} networks, where the degrees typically remain bounded. Quantities of interest are the mixing times and the cover times of random walks on these networks, under different types of dynamics. These enable the description of propagation of information through the network (see, for example, \cite{Lesk08}, \cite{LevPer17}, and references therein).

\cite{Crane2016} probably contains the first mathematically rigorous attempt to capture the limiting dynamics of time-varying networks as the number of edges per vertex grows to infinity. In the context of dense graph limit theory, initiated by \cite{Lovasz2006}, these limits can be understood by means of \emph{graphons}, which can be turned into a compact metric space. While it is intuitively (and mathematically) easy to construct random dynamics of graphs on~$n$ vertices for each fixed~$n$, it is non-trivial to realise them in such a way that the dynamics remain visible in the limit as~$n\to\infty$. This is because, as the time-evolving dense graph sequence approaches the appropriate evolving graphon, a lot of averaging takes place that typically results in a \emph{deterministic} flow. 

Crane's starting point is the Aldous-Hoover theory for infinitely exchangeable arrays. Let~$\cG_\infty$ be the space of infinite arrays equipped with the product topology. Define a modulus map~$\abs{\cdot}$ that takes an array~$\Gamma\in\cG_\infty$ to a graphon~$\abs{\Gamma}$ (which is well-defined with probability~$1$ when the array is exchangeable). One of the main results of \cite{Crane2016} is that a~$\cG_\infty$-valued exchangeable Markov process~$(\Gamma(s))_{s \geq 0}$ induces a graphon-valued Markov process~$(\abs{\Gamma(s)})_{s \geq 0}$, and that the latter has locally bounded variation. Consequently, this route only leads to jump processes and deterministic flows on the space of graphons, \emph{not} to diffusion-like processes (see \cite{Crane2016}, \cite{CK18}).

However, this does \emph{not} imply that there \emph{are} no diffusion-like processes on graphons. In fact, as we show in what follows, the limitation is imposed by the theory of infinitely exchangeable arrays, not by the theory of graphons. More precisely, we start with the Aldous-Hoover theory and represent an~$[0,1]$-valued infinitely exchangeable array~$(X_{ij})_{i,j\in\N}$ as
\begin{equation}\label{1}
X_{ij} = f(U,U_i,U_j,U_{ij}), \qquad i,j \in \N\Def\{1,2,\dots\},
\end{equation}
for some function~$f\colon\,[0,1]^4\to [0,1]$, where~$U$, $(U_i)_{i\in\N}$ and~$(U_{ij})_{i,j\in\N}$ are independent and identically distributed uniform random variables. If we assume that~$(U(s))_{s \geq 0}$, $(U_i(s))_{s \geq 0}$ and~$(U_{ij}(s))_{s \geq 0}$ are Markov processes having the uniform distribution as their equilibrium distribution, then we can construct from~\eq{1} a~$\cG_\infty$-valued process
\begin{equation*} 
X_{ij}(s) = f\big(U(s),U_i(s),U_j(s),U_{ij}(s)\big),
\qquad i,j \in \N.
\end{equation*}
It is clear that, in general, $(X_{ij}(s))_{s \geq 0}$ is not Markov, since functions of Markov processes need not be Markov under the filtration of the mapped process. Moreover, even if it is Markov, then once the process is projected to the graphon space many of its properties are lost.  Below we will illustrate that interesting Markov processes on the graphon space need not be Markov on~$\cG_\infty$, and thus are \emph{not captured through the lens of the Aldous-Hoover theory}.


\subsection{A short primer on graphon-valued stochastic processes}

The approach we take in this paper is to work directly with networks and their graphon limits, and we illustrate this by means of a classical model from population genetics. Concretely, we assign types to each individual and define connection probabilities based on the types. Then we impose the dynamics on the types from population genetic models to enable a scaling limit that results in diffusive dynamics on the space of graphons. The type space is allowed to be continuous or discrete. This allows us to also observe the dynamics as a rescaled limit of a time-evolving finite dense graph sequence, thus providing a natural justification for the dynamics. We present this approach first in the case where each individual is one of two types and all individuals within each type are connected with each other via the below example.

The following example shows that graphon-valued diffusions can arise from discrete models.

\begin{example}  
Consider~$n$ individuals, each carrying Type~$0$ or Type~$1$. Suppose that each individual, independently and at rate~$1$, randomly draws an individual from the population (possibly itself) and adopts its type. Let~$X^n(s)$ be the number of individuals of type~$0$ at time~$s$. For each~$s \geq 0$, think of the~$n$ individuals as the vertices of a random graph~$G^n(s)$ in which individuals~$i$ and~$j$ are connected by an edge with probability~$1$ if they are of the same type and remain disconnected if their types are different. Using~\eq{1b} we see that, for any connected graph~$F$ on~$k$ vertices, the subgraph density of~$F$ in~$G^n(s)$ is
\baqn{
t_F(G^n(s)) &\Def  
\frac{\mbox{\# of copies of~$F$ in~$G^n(s)$}}
{\mbox{\# of copies of~$F$ in the complete graph}}\nonumber\\
&= \frac{X^n(s)^k + \bclr{n-X^n(s)}^k}{n^k} 
\label{4}
} 
(we will give a rigorous definition of~$t_F$ later). If~$F$ consists of multiple components, then~$t_F$ is just the product of~$t_{F_i}$ with~$F_i$ the individual connected components of~$F$. Let~$Y^n(s) = \frac{1}{n} X^n(ns)$ represent the fraction of individuals of Type~$1$ in the population at time~$s$ on time scale~$n$. It is well known that if~$Y^n(0)$ converges weakly to~$Y(0)$, then~$Y^n=(Y^n(s))_{s \geq 0}$ converges weakly to~$Y=(Y(s))_{s \geq 0}$ in path space with respect to the Skorohod topology as~$n\to\infty$, where the limiting process is the Wright-Fisher diffusion on~$[0,1]$, given by the SDE 
\beq{
Y(s) = Y(0) + \int_0^s \sqrt{Y(u)(1-Y(u))}\,\dd W(u)
} 
with initial condition~$Y(0)$ and with~$W=(W(s))_{s \geq 0}$ being standard Brownian motion. Clearly,~\eq{4} implies that
\begin{equation}\label{100}
\lim_{n\to\infty} t_F(G^n(ns)) = Y(s)^k + (1-Y(s))^k.
\end{equation}
We observe that the right-hand side of~\eq{100} equals~$t_F(\tilde{h}(s))$, the subgraph density corresponding to the graphon~$\tilde{h}(s)$ drawn in Figure~\ref{fig1}. It is also easily seen that~$t_F(\tilde{h}_s)$ is adapted to the filtration generated by~$Y_s$, and is a Markov process. Furthermore, we can calculate the modulus of continuity of~$t_F(\tilde{h})$ and conclude in the subgraph distance~$\tilde{h}$ is diffusive. Consequently, we have constructed a graphon-valued diffusion and a sequence of finite graph-valued processes that converges to it (see Section~\ref{s:graphons} for precise definitions).\hfill$\square$
\end{example}

\def\stka#1#2#3{\begin{tikzpicture}[scale=0.35]
	\def\s{#1}
	\def\sm{#2}
	\coordinate (e1) at (10,0);
	\coordinate (e2) at (0,10);
 	\coordinate (u1) at ($(e1)+(e2)$);
	\coordinate (u2) at ($\s*(e1) + \sm*(e2)$);
	\coordinate (u3) at ($\s*(e2) + \sm*(e1)$);
	\draw (0,0) -- (e1) -- ($(e1) + (e2)$) -- (e2) -- (0,0);
	\draw ($\s*(e2)$) -- ($\s*(e2) + (e1)$);
	\draw ($\s*(e1)$) -- ($\s*(e1) + (e2)$);
	\draw ($.5*\s*(u1)$) node {$1$};
	\draw ($\s*(u1) + .5*\sm*(u1)$) node {$1$};
	\draw ($\s*(e2) + .5*(u2)$) node {$0$};
	\draw ($\s*(e1) + .5*(u3)$) node {$0$};
	\draw ($\s*(e1)$) node[yshift=-13pt] {#3};
\end{tikzpicture}
}
\begin{figure}
\centering
\stka{.7}{.3}{ $Y(s)$}
\enskip
\caption{\label{fig1} \small
Graphical representation of the limiting graphon-valued stochastic process arising from a simple dynamical graph model, with~$(Y(s))_{s \geq 0}$ the Wright-Fisher diffusion.}
\end{figure} 

\hspace{0.4cm}
The next example shows that interesting and natural processes can be constructed in such a way that they are not Markov when seen through the Aldous-Hoover theory of infinitely exchangeable arrays, but are Markov when projected to graphons. 

\begin{example}
Recall the notation~$\{X_{ij}\colon\, i,j \in \N\}$ used for exchangeable arrays earlier. Consider the Wright-Fisher diffusion~$(Y(s))_{s\geq 0}$ on~$[0,1]$, and put
\begin{equation*}
X_{ij}(s) = \I[U_{ij}\leq Y(s)], \qquad \text{$i,j \in \N$, $s \geq 0$,}
\end{equation*}
where the~$U_{ij}$ are independent and identically distributed uniform random variables. This process is non-Markov in~$\cG_\infty$ with respect to its filtration
\begin{equation*}
\cF(s)=\sigma(X_{ij}(u)\colon\, 0\leq u\leq s,\,i,j\in\N).
\end{equation*}
Although~$Y(s)$ is measurable with respect to~$\cF(s)$, because
\begin{equation*}
\lim_{n\to\infty} \frac{2}{n(n-1)} \sum_{1\leq i <  j \leq n} X_{ij}(s) 
=Y(s) \qquad \text{almost surely},
\end{equation*}
the individual~$U_{ij}$ are \emph{not}. Still, the projected graphon-valued process is~$(Y(s))_{s \geq 0}$, which \emph{is} Markov with respect to its filtration~$\cF(s)=\sigma(Y_u\colon\, 0\leq u\leq s)$. \hfill~$\square$
\end{example}

\hspace{0.4cm}
In view of the above, it seems that the Aldous-Hoover theory, while being a natural starting point, is ultimately not the right way to develop a theory of graphon processes. Being Markov on the space of infinitely exchangeable arrays is too strong a condition, since it restricts the possible dynamics in the graphon space to those with locally bounded variation.


\subsection{Outline}

The remainder of the paper is organised as follows. In Section~\ref{4.5}, we present the preliminaries required. We provide a brief introduction into the theory of graphons in Section~\ref{s:graphons} and a quick review of population models in Section~\ref{s:pop}. In Section~\ref{5}, we present our main results. We begin in Section~\ref{6} by discussing the Skorohod topology on the graphon space and provide a framework for understanding weak convergence in graphon space by means of sub-graph densities (Theorem~\ref{thm1} and Corollary~\ref{cor5}). In Section~\ref{9}, we show that the models discussed in Section~\ref{s:pop} provide a natural class of graphon dynamics: the graphon is obtained by connecting pairs of individuals with a probability that is given by a type-connection matrix, whose entries depend on the empirical type distribution of the entire population via a fitness function (Theorem~\ref{13}). After that we let the number of types tend to infinity and arrive at a graphon dynamics governed by the Fleming-Viot diffusion (Theorem~\ref{17}). In Section~\ref{ex1}, we give concrete examples to illustrate the abstract results in Section~\ref{5} (Examples~\ref{exam1}--\ref{exam3}) and offer some remarks on possible generalisations. In Section~\ref{18}, we give the proofs of the theorems.  


\section{Preliminaries} 
\label{4.5}

We begin this section by stating the minimally required preliminaries on dense graphs, graphons, their equivalence classes, and the metric space they belong to.


\subsection{Graphons} 
\label{s:graphons} 

We provide a brief introduction of the preliminaries of graphons and construct an example of a simple diffusion on the space of
graphons. Let~$\cG_n$ be the set of all graphs on~$n$ vertices. Graphs on~$n$ vertices with order~$n^2$ edges are referred to as \emph{dense graphs}. For any two graphs there is a natural definition of distance between them, given by the subgraph distance. More precisely, if~$F$ is a simple graph on~$k$ vertices and~$G$ is a graph on~$n$ vertices, then the \emph{subgraph density} is defined as
\begin{equation}\label{1b}
t_{F}(G) \Def \frac{\abs{\hom(F,G)}}{n^k} \in[0,1],
\end{equation} 
where~$\hom(F,G)$ denotes the set of homomorphisms from~$F$ to~$G$.\footnote{Recall that a homomorphism from a graph~$F$ to a graph~$G$ is a function that maps the vertices of~$F$ to the vertices of~$G$ in such a way that edges are mapped to edges.}
 
Let~$\cF$ denote the set of isomorphism classes of finite graphs given by~$\cF=\{F_i\}_{i\in\N}$, with each~$F_i$ being a representative of an isomorphism class. We can then define the \emph{subgraph distance} of two graphs~$G_1$ and~$G_2$ as
\beq{
\dsub(G_1,G_2) \Def  \sum_{i\in\N} 2^{-i}|t_{F_i}(G_1) - t_{F_i}(G_2)|. 
}
This metric has some nice properties. For instance, it is known that~$(\cF,\dsub)$ is a discrete metric space and that the completion of~$\cF$ with respect to this metric is given by the space~$\cal W$, which is the space of measurable functions~$h\colon\,[0,1]^2 \to [0,1]$ satisfying~$h(x,y) = h(y,x)$ for all~$(x,y) \in [0,1]^2$. The elements of~$\cal W$ are called \emph{graphons}. The definition of~$\dsub$ can be extended to graphons. For~$h \in {\cal W}$ and~$F$ a simple finite graph on~$k$ vertices, we let
\begin{equation}\label{2}
t_{F}(h) \Def \int_{[0,1]^k}  \prod_{\{i,j\} \in E(F)} h(x_i,x_j)\, d x_1 \cdots d x_k.
\end{equation}

One of the key results in dense graph theory the above definition is based on is the following theorem.

\begin{theorem}\label{THM0} 
Let~$(G_n)_{n\geq 1}$ be a dense-graph sequence that is Cauchy with respect to~$\dsub$. Then there exists an~$h \in \cal W$ such that
\begin{equation*} 
\dsub\klr{G_n,h} \to 0 \quad (n \to \infty). 
\end{equation*}
\end{theorem}

For a proof of the above see \cite{Lovasz2006}, which uses Szemer\'edi partitions and the Martingale Convergence Theorem, or \cite{Diaconis2008}, who show that it can be proved by using results from \cite{Hoover1979} and \cite{Aldous1981}. Note that~$h$ above is in general not unique, but this will not be of importance for what follows; we refer to \cite{Borgs2008} for a discussion of this and related questions. We refer to \cite[Chapter 11]{Lovasz2012} for a detailed discussion of convergence of dense graph sequences.

A convenient way of ``creating'' finite (random) graphs on~$n$ vertices from a  standard kernel~$h$ is the following model, which we will denote by~$G(n,h)$. Firstly, let~$U_1,\dots,U_n$ be i.i.d.\ with uniform distribution on~$[0,1]$. Secondly, for each two vertices~$i$ and~$j$, connect them with probability~$h(U_i,U_j)$, independently of all the other edges. It is not difficult to prove that 
\ben{\label{1a} 
\dsub\bklr{G(n,h),h} \to 0 \quad \text{ almost surely} \quad (n \to\infty).
}
This is, in some sense, the basic law of large numbers in dense graph theory. In this paper, instead  of sampling the labels i.i.d.\ and uniformly from~$[0,1]$, we will allow the labels to be sampled in a more general way.

We can interpret~\eq{2} as the normalized number of homomorphisms of~$F$ into a weighted graph with the uncountable vertex set~$[0,1]$, with edge weights given by~$h$. Furthermore, any finite simple graph~$G$ on~$n$ vertices can be represented canonically by a graphon via
\begin{equation}\label{3}
h^{G}(x,y) \Def 
\begin{cases}
1 &\text{if there is an edge between vertex~$\lceil{nx}\rceil$  and vertex~$\lceil{ny}\rceil,$}\\
0 &\text{otherwise.}
\end{cases}
\end{equation}
We easily verify that
\beq{
t_F(G) = t_F\bclr{h^{G}},
}
so that all definitions are consistent. The representation of a graph via a graphon is not unique. Indeed, there are graphs that have the same graphon representation (for example, all complete graphs have the same representation~$h^G\equiv 1$). In the context of graph limit theory this is not an issue, as long as we assume that the number of vertices of the graph sequence tends to infinity. Moreover, the representation in~\eq{3} depends on the ordering of the vertices. Since typically we are interested only in graph properties that are independent of vertex labels (for example subgraph densities), it is natural to consider equivalence classes of graphons obtained by letting two graphons be equivalent when they are identical ``up to vertex labels''. To make this rigorous, let~$\Sigma$ be the space of measure-preserving bijections~$\sigma\colon\, [0,1] \to [0,1]$. Then~$h_1\equiv h_2$ if there is a~$\sigma\in\Sigma$ such that 
\begin{equation*}
h_1(x,y) = h_2(\sigma x, \sigma y), \qquad x,y\in[0,1].
\end{equation*}
The equivalence relation yields the quotient space~$\widetilde{\cal W}$. It is known that~$(\widetilde{\cal W},\dsub)$ is a compact metric space, and therefore complete and separable (see, for example, \cite[p.\,695]{Crane2016} or \cite[Section 2]{Bollobas2009}, although it is not difficult to deduce this fact from the fundamental results of \cite[Theorems~9.23 and~11.5]{Lovasz2012}).

The reader is referred to \cite[Section 2]{Bollobas2009}, \cite{Borgs2008}, \cite{Lovasz2012} for a structured and more detailed exposition on dense graphs and graphons. Below is a simple example where we see a diffusion arising naturally on graphon space.


\subsection{Population dynamics} 
\label{s:pop}

For obtaining dynamics on graphons we will use models from population biology. We provide a very brief and quick review of the literature. The multi-type Moran model has~$m+1$ types, which are labelled~$0,\ldots,m$. Consider~$n$ individuals, each carrying one of the types from~$\{0,\ldots,m\}$. Suppose that each individual, independently and at rate~$1$, randomly draws an individual from the population (possibly itself) and adopts its type. Let~$X_\ell^{m,n}(s)$ denote the number of individuals of type~$\ell$ at time~$s$, where~$0\leq\ell\leq m-1$, and let
\beq{ 
X^{m,n}(s)= \bclr{X^{m,n}_0(s),\dots,X^{m,n}_{m-1}(s)}
} 
be the corresponding vector of type counts. For convenience, we also define the number of individuals of type~$m$ at time~$s$ to be~$X^{m,n}_m(s) \Def n-\sum_{\ell=0}^{m-1} X^{m,n}_\ell(s)$. Whenever we consider the entire process, we will drop time and write~$X^{m,n}_\ell$ for the individual counting processes, or~$ X^{m,n}$ for the multivariate counting process.

Consider the space-time rescaling
\beq{
Y^{m,n}(s) \Def \frac{1}{n} X^{m,n}(ns), \quad s \geq 0,
}
which consists of~$m$ components
\begin{equation}
Y^{m,n}(s)= \bclr{Y^{m,n}_0(s),\dots,Y^{m,n}_{m-1}(s)},
\end{equation}
representing the fractions of individuals of types~$0,\ldots,m-1$ at time~$ns$. Analogously to before, the fraction of individuals 
of type~$m$ at time~$ns$ is denoted by~$Y^{m,n}_m(s) \Def 1-\sum_{\ell=0}^{m-1} Y^{m,n}_\ell(s)$. It is known (see \cite[Section 2]{Dawson1993}) that if 
\beqn{\label{10}
Y^{m,n}(0) \wconv Y^{m}(0)\quad(n\toinf),
}
then 
\beqn{\label{11}
Y^{m,n} \wconv Y^m\quad (n\toinf),
}
(keep in mind that~\eq{11} states weak convergence at the process level). The limiting process consists of~$m$ components
\begin{equation*}
Y^{m}(s)= \bclr{Y^{m}_0(s),\dots,Y^{m}_{m-1}(s)},
\end{equation*}
taking values in the~$m$-dimensional simplex 
\begin{equation*}
\cS^m  = \Big\{x = (x_0,\dots,x_{m-1})\in\RR^m\colon\, \text{$x_\ell \geq 0$ for all~$0 \leq \ell \leq m-1$, $\textstyle\sum_{\ell=0}^{m-1} x_\ell \leq 1$}\Big\}, 
\end{equation*}  
and is referred to as the \emph{Wright-Fisher diffusion}. For convenience, we let~$Y^m_m(s) = 1-\sum_{\ell=0}^{m-1} Y^m_\ell(s)$. We will also need the cumulative distribution function of the type distribution, defined as
\begin{equation}\label{11b}
F^m(s; x) \Def  \sum_{\ell=0}^{\lfloor (m+ 1)x \rfloor} Y^m_\ell(s), \qquad x \in [0,1),
\end{equation}
and~$F^m(s; 1)=1$. Note that~$x \mapsto F^m(s;x)$ can have jumps at~$x \in \{0,\tfrac{1}{m+1},\dots,\tfrac{m}{m+1}\}$.

Recalling the definition of~$F^m$ at~\eq{11b}, we can define the corresponding empirical type distribution
\begin{equation*}
Z^m(s) = \sum_{\ell=0}^m Y^m_\ell(s)\,\delta_{\ell/(m+1)}.
\end{equation*}
Note that
\begin{equation*}
Z^m(s;[0,x]) = F^m(s;x), \qquad x \in [0,1].
\end{equation*}
It is known (see \cite[Section 2]{Dawson1993}) that if
\begin{equation}\label{16}
Z^m(0) \wconv Z(0)\quad(m\toinf),
\end{equation}
then
\begin{equation*}
Z^m \wconv  Z\quad(m\toinf).
\end{equation*}
The limiting process~$Z$ takes values in~$\cP([0,1])$, the set of probability measures on~$[0,1]$ endowed with the topology of weak convergence, and is referred to as the \emph{Fleming-Viot diffusion}. We will also need the process of cumulative distribution function of~$Z$, defined by
\begin{equation*}
F(s;x) \Def Z(s;[0,x]), \qquad x \in [0,1].
\end{equation*}


\section{Main results} 
\label{5}

On the metric space~$(\widetilde{\cal W},\dsub)$, we can define the Skorohod topology on~$\widetilde{\cal W}$-valued paths in the usual way; see for example  or \cite{Billingsley1999,Ethier1986}. Denote by~$D=D([0,\infty),\widetilde{\cal W})$ the set of c\`adl\`ag paths in~$\widetilde{\cal W}$, which can be equipped with a metric~$d^\circ$, turning~$D$ into a complete and separable metric space since~$(\widetilde{\cal W},\dsub)$, being compact, is complete and separable. We use ``$\asconv$`` to denote convergence with respect to the underlying metric space, and we use ``$\wconv$'' to denote weak convergence with respect to the Borel-sigma-algebra induced by that metric. Note that we will use ``$\asconv$`` and ``$\wconv$'' also for convergence, respectively, weak convergence in~$(\widetilde{\cal W},\dsub)$ itself. Let~$\tilde{h}$ be a~$\widetilde{\cal W}$-valued stochastic process. We write~$\tilde{h}(s)$ to denote the value of the process a time~$s\geq 0$, which is an equivalence class of graphons. If~$h(s)$ is a representative graphon of the equivalence class~$\tilde h(s)$, we write~$h(s; x,y)$ to denote the value of that graphon evaluated at coordinates~$(x,y)\in[0,1]^2$. Note that, for a given~$\tilde{h}\in D$ and a given simple finite graph~$F$, we can consider the real-valued process~$t_F(\tilde{h})$ as an element of~$D([0,\infty),[0,1])$. We write~$t_F(\tilde{h}(s))$ to denote the value of the process a time~$s\geq 0$.
   
    
\subsection{Weak convergence of graphons} 
\label{6}

Let~$(\tilde{h}^n)_{n\in\N}$ be a sequence of~$\widetilde{\cal W}$-valued stochastic processes. In order to prove weak convergence in the graphon space, that is, $\tilde{h}^n\wconv \tilde{h}$ in~$D$, we need to establish (as in the case of processes on any metric space):
\begin{enumerate}
\item[$(i)$] 
Convergence of finite-dimensional distributions~$(\tilde{h}^n(s_i))_{1\leq i\leq d}\wconv (\tilde{h}(s_i))_{1\leq i\leq d}$ for all points~$s_1,\dots,s_d\geq 0$ at which~$\tilde{h}$ is continuous almost surely.
\item[$(ii)$] 
Tightness of the sequence~$(\tilde{h}^n)_{n\in\N}$.
\end{enumerate}
Our first theorem provides equivalent criteria for establishing weak convergence by means of the corresponding subgraph density processes. In the statement and proof, $\IP$ will denote the probability measure under which the expectation~$\IE$ is taken.

\begin{theorem}\label{thm1} 
Let~$\tilde{h}$ and~$(\tilde{h}^n)_{n\in\N}$ be random elements in~$D([0,\infty),\widetilde{\cal W})$. Then the following are equivalent:
\begin{enumerate}
\item[$(i)$] 
$\tilde{h}^n\wconv \tilde{h}$ as~$n\toinf$.
\item[$(ii)$] 
For all~$d\geq 1$ and all graphs~$F_1,\dots,F_d\in\cF$,
\begin{equation}\label{6a}
\bclr{t_{F_1}(\tilde{h}^n),\dots,t_{F_d}(\tilde{h}^n)}\wconv\bclr{t_{F_1}(\tilde{h}),\dots,t_{F_d}(\tilde{h})} 
\quad (n\to\infty).
\end{equation}
\item[$(iii)$] 
For every graph~$F\in\cF$, the sequence~$(t_F(\tilde{h}^n))_{n\in\N}$ is tight and, for all~$k\geq 1$, all real numbers~$0\leq s_1<\dots < s_d<\infty$ where~$\tilde{h}$ is continuous almost surely, and all graphs~$F_1,\dots,F_d\in \cF$, 
\beqn{\label{7}
\lim_{n\toinf}\EE\bclc{t_{F_1}(\tilde{h}^n(s_1))\cdots t_{F_d}(\tilde{h}^n(s_d))} = \EE\bclc{t_{F_1}(\tilde{h}(s_1))\cdots t_{F_d}(\tilde{h}(s_d))}.
}
\end{enumerate}
\end{theorem}

For finite graphs it is typically easier to work with \emph{injective} homomorphisms, the set of which we denote by~$\inj(F,G)$. If~$F$ has~$k$ vertices, then for~$n\ge k$ there are at most~$n_{(k)}\Def n(n-1)\dots (n-k+1)$ such mappings, which provides a standard normalisation to the count of injective homomorphisms. Thus, if~$F$ is a graph on~$k$ vertices and~$G$ a graph on~$n$ vertices, then we define
\begin{equation*}
t^\inj_{F}(G) \Def \frac{\abs{\inj(F,G)}}{n_{(k)}} \in[0,1]
\end{equation*}
if~$k\leq n$ and~$t^\inj_{F}(G)\Def0$ otherwise. It is easy to see that
\beqn{\label{8}
\babs{t^{\inj}_{F}(G)-t_{F}(G)} \leq \frac{C_F}{n}
}
for some constant~$C_F$ that only depends on~$F$. So, for most purposes the two objects are equivalent in the limit~$n\toinf$.

We say that~$G:= (G(s))_{s\geq 0}$ is a \emph{graph process} if the induced graphon process~$h^{G}:= \clr{h^{G(s)}}_{s\geq 0}$ is a random element of~$D([0,\infty),\widetilde{\cal W})$. Similarly as before, we can consider the process~$t^\inj_F(G)$ as an element of~$D([0,\infty),[0,1])$. In order to simplify notation, for a sequence of graph processes~$(G_n)_{n\in\N}$ we will write~$G_n\wconv \tilde{h}$ instead of~$h^{G_n}\wconv \tilde{h}$. We can use~\eq{8} to prove the following corollary of Theorem~\ref{thm1}.

\begin{corollary}\label{cor5} 
Let~$(G_n)_{n\in\N}$ be a sequence of graph processes such that 
\beq{
\inf_{s\geq 0}\abs{V(G_n(s))}\toinf\quad (n\toinf),
}
and let~$\tilde{h}$ be a random element in~$D([0,\infty),\widetilde{\cal W})$. Then the following are equivalent:
\begin{enumerate}
\item[$(i)$] 
$G_n\wconv \tilde{h}$ as~$n\toinf$.
\item[$(ii)$] 
For all~$d\geq 1$ and all graphs~$F_1,\dots,F_d\in\cF$,
\begin{equation}\label{7c}
\bclr{t^\inj_{F_1}(G_n),\dots,t^\inj_{F_d}(G_n)}\wconv
\bclr{t_{F_1}(\tilde{h}),\dots,t_{F_d}(\tilde{h})} \quad (n\to\infty).
\end{equation}
\item[$(iii)$] 
For every graph~$F\in\cF$, the sequence~$(t^\inj_F(G_n))_{n\in\N}$ is tight and, for all~$d\geq 1$, all real numbers~$0\leq s_1<\dots < s_d<\infty$ where~$\tilde{h}$ is continuous almost surely, and all graphs~$F_1,\dots,F_d\in \cF$,
\beqn{\label{7b}
\lim_{n\toinf}\EE\bclc{t^\inj_{F_1}(G_n(s_1))\cdots t^\inj_{F_d}(G_n(s_d))} = \EE\bclc{t_{F_1}\bclr{\tilde{h}(s_1)}\cdots t_{F_d}\bclr{\tilde{h}(s_d)}}.
}
\end{enumerate}
\end{corollary}

\begin{proof} 
Equivalence of~\eq{6a} and~\eq{7c} follows from~\eq{8} and Slutsky's Theorem. Equivalence of~\eq{7} and~\eq{7b} is immediate from~\eq{8} and the bounded convergence theorem, whereas equivalence of tightness in~$(iii)$ of Theorem~\ref{thm1} and~$(iii)$ of Corollary~\ref{cor5} is a consequence of~\eq{8}, \cite[Inequality~(6.3), p.\,122]{Ethier1986} and \cite[Corollary~7.4, p.\,129]{Ethier1986}. 
\end{proof}


\subsection{Graphon dynamics}
\label{9}

We will use the population models described in Section~\ref{s:pop} to construct a sequence of random graphs that evolve in time and converge weakly in the space of graphons.

Consider the interval~$[0,1]$ with the Euclidean metric, and let~$\cL = C([0,1],[0,1])$, the space of continuous functions from~$[0,1]$ to~$[0,1]$ endowed with the uniform topology, which we will call the space of \emph{fitness landscapes}.  Note that, in the definition of~$\cL$, the first appearance of~$[0,1]$ represents a space of types, and the second appearance of~$[0,1]$ represents some sort of \emph{fitness} associated with each type through a function from~$\cL$.

Here are some specifications that we will assume throughout this section.

\begin{itemize}
\item[{\bf (R1)}] The connection probabilities between individuals depend on the types and fitness. They will be given by a continuous function~$r\colon\,[0,1]\times [0,1]\rightarrow [0,1]$. 
\item[{\bf (H1)}] The fitness landscape changes dynamically over time. To represent this, for~$m,n \in \N$, we consider~$H^{m,n}$, $H^{m}, H$ to be random elements in~$D([0,\infty), \cL)$, which will be referred to as \emph{fitness landscape processes}. We write~$H^{m}(s)$ for the fitness landscape at time~$s$, and~$H^{m}(s; x)$ for the fitness of type~$x$ at time~$s$.
\end{itemize}

\noindent In what follows, we denote by~$\bar F$ the right-continuous generalised inverse of a distribution function~$F$ with support~$[0,1]$, defined in the usual way as
\begin{equation*}
\bar{F}(u) \Def \inf\{x\in[0,1]\colon\, F(x)>u\}, \qquad u\in [0,1),
\end{equation*} 
and, for convenience, we set~$\bar F(1)\Def\lim_{x\uparrow1} \bar F(x)$.

\paragraph{Wright-Fisher dynamics.}

Our first construction of a graphon dynamics comes from the scaling limit of the multi-type Moran model to the multi-dimensional
Wright-Fisher diffusion. As discussed in Section~\ref{s:pop}, the multi-type Moran model has~$m+1$ types, which are labelled
$0,\ldots,m$. Consider~$n$ individuals, each carrying one of the types from~$\{0,\ldots,m\}$. Let~$\tau_i^{m,n}(s)$ denote the type of individual~$i$ divided by~$m+1$ at time~$ns$; that is, if at time~$ns$ the type of individual~$i$ equals~$k$, where~$0\leq k\leq m$, then~$\tau_i^{m,n}(s)=k/(m+1)$.

\medskip\noindent
\emph{Discrete graphon dynamics:} We can construct the graph~$G^{m,n}(s) \in \cG_n$ at time~$s \geq 0$ by connecting~$i$ and~$j$ if
\begin{equation}\label{12}
U_{ij}^n <  r\big(H^{m,n}(s; \tau_i^{m,n}(s)),H^{m,n}(s; {\tau_j^{m,n}(s)})\big),
\end{equation}
with~$\{U^n_{ij}\colon\, n \in \N,\,1 \leq i < j \leq n\}$ being a collection of independent uniform random variables on~$[0,1]$,
independent of everything else. Thus, $(G^{m,n})_{n \in \N}$ is a sequence of graph processes, which evolve in time due to the induced dynamics of the Moran model on~$m+1$ types. Note that we have scaled time in such a way that the graph~$G^{m,n}(s)$ represents the situation of the underlying Moran model at time~$ns$. 

\begin{theorem}\label{13} Let~$G^{m,n}$ be constructed as above. Assume {\bf (R1)}, {\bf (H1)}, and suppose that~$\clr{Y^{m,n},H^{m,n}}\wconv\clr{Y^{m},H^{m}}$ as~$n\toinf$. Then
\begin{equation*}
G^{m,n}\wconv \tilde h^m\quad(n\toinf),
\end{equation*}
where, for each~$s\geq 0$, the equivalence class~$\tilde h^m(s)$ has a representative~$h^m(s)$ of the form
\begin{equation}\label{14}
h^m(s;x,y) = r\bclr{H^m (s;\bar{F}^m(s;x) ), H^m (s;\bar{F}^m(s;y) )},  
\qquad (x,y) \in [0,1]^2.
\end{equation} 
\end{theorem} 

While the space of types~$[0,1]$ allows for uncountably many types for convenience, in the example of the~$(m+1)$-Moran model, we will embed the~$m+1$ types into~$[0,1]$ in a canonical way, using the functions~$\tau_i^{m,n}(s)$ For both the type space and the fitness space, the interval~$[0,1]$ could be replaced by any other Polish space, as is the case for graphons.

In~$(G^{m,n})_{n \in \N}$, the probability for two vertices to be connected depends on their respective fitness, which in the case
of~\eq{12} depends on their types.  The limiting graphon is an~$(m+1) \times (m+1)$-block graphon whose block boundaries move along the diagonal according to the Wright-Fisher diffusion and whose block heights are controlled by the fitness landscape. As observed for the case~$m=1$ in Section~\ref{intro}, it is not too hard to see that indeed the limiting graphon is a diffusion on the graphon space.


\paragraph{Fleming-Viot dynamics.}

In the above dynamics the resulting graphon process falls in the class of stochastic block models. Our second theorem provides a graphon dynamics that arises from the Fleming-Viot diffusion and has no block structure in the limit.

\begin{theorem}\label{17} Assume {\bf (R1)}, {\bf (H1)}, and suppose that  $(Z^m,H^m)\wconv(Z,H)$ as~$m\toinf$. Let~$\tilde{h}^m$ be defined through its representative~$h^m$ as given in~\eq{14}.  Then
\begin{equation*}
\tilde h^m \wconv \tilde h \quad(m\toinf),
\end{equation*}
where, for each~$t\geq 0$, the equivalence class~$\tilde h(s)$ has a representative~$h(s)$ of the form
\begin{equation}\label{14b}
h(s;x,y) = r\bclr{H (s; \bar{F}(s;x) ), H (s;\bar{F}(s;y))},  
\qquad (x,y) \in [0,1]^2.
\end{equation}
\end{theorem}

Our~$H$ can be any fitness landscape process in {\bf (H1)} and~$r$ any graphon in {\bf (R1)}. We will show in Example~\ref{exam3} below that these allow for a general class of graphons beyond the stochastic block model process observed earlier.

Also note that the appearance of the inverse distribution function in~\eq{14}, respectively~\eq{14b}, can be understood as a change of reference measure of the vertex space from uniform to~$Z^m$, respectively~$Z$. For example, with~$h^m$ as in~\eq{14}, subgraph densities with respect to this graphon can be written as
\beqs{
t_F(h^m(s)) & = \int_{[0,1]^k} \prod_{\{i,j\}\in E(F)}h^m(s;x_i,x_j)\, \dd x_1\dots\dd x_k \\
& = \int_{[0,1]^k} \prod_{\{i,j\}\in E(F)}r(H^m(s,x_i),H^m(s,x_j))\, Z^m(s;\dd x_1)\cdots Z^m(s;\dd x_k),
}
and similarly for~$ t_F(h(s))$ with~$h$ as in~\eq{14b}. This effect was already observed by \cite{Athreya2016} in the context of Respondent Driven Sampling.


\section{Examples and Remarks}
\label{ex1}

In this section we illustrate the abstract results in Section~\ref{5} via concrete examples of type-connection graphons and fitness landscapes. 


\subsection{Three examples}

\begin{example}\label{exam1}
Let
\begin{equation*}
H^{m,n}(s;u) =H^m(s;u)=H(s;u) =u. 
\end{equation*}
It follows that~$G^{m,n}(s) \in \cG_n$ is the random graph in which, at time~$s$, individuals~$i$ and~$j$ are connected by an edge with probability~$r\bclr{\tau^{m,n}_i(s),\tau^{m,n}_j(s)}$, where again~$\tau^{m,n}_i(s)$ denotes the type of individual~$i$ divided by~$m$ at time~$ns$. By Theorem~\ref{13}, as~$n \to \infty$, the graph process~$G^{m,n}$ converges to~$\tilde{h}^m(s)$ given by~\eq{14}. Thus, individuals~$i$ and~$j$ are connected with a probability that \emph{does not} depend on the rest of the population. In this case, $r$ will only ever be evaluated at the points~$\{0,\tfrac{1}{m+1},\dots,\tfrac{m}{m+1}\}\times\{0,\tfrac{1}{m+1},\dots,\tfrac{m}{m+1}\}$, an so~$r$ can be viewed as an~$(m+1)\times(m+1)$ type-connection matrix~$(r_{ij})_{0\leq i,j\leq m}$ (the matrix can be extended to a function on~$[0,1]^2$ by an arbitrary continuous interpolation). This leads to the representation
\beq{
 h^m(s;x,y) = r_{(m+1)\bar F^m(s;x),(m+1)\bar F^m(s;y)}.
}
In particular, for the case~$m=1$ (two types), the graphon~$r$ can, without loss of generality, be assumed to be given by a~$(2 \times 2)$-matrix, namely,
\begin{equation*}
  r  = \left[\begin{array}{cc}\alpha & \delta \\ \delta&\beta\end{array}\right],
\end{equation*}
which yields the representation
\begin{equation*}
h^1(s;x,y)
=\begin{cases}
\alpha &\text{if~$(x,y) = [0,Y_0^1(s)) \times [0,Y_0^1(s))$,}\\[.5ex]
\beta &\text{if~$(x,y) = [Y_0^1(s),1] \times [Y_0^1(s),1]$,}\\[.5ex]
\delta &\text{if~$(x,y) =  [0,Y_0^1(s)) \times [Y_0^1(s),1]$,}\\[.5ex]
\delta &\text{if~$(x,y) =  [Y_0^1(s),1] \times [0,Y_0^1(s))$.}
\end{cases}
\end{equation*}
We have convergence to a graphon-valued Markov process that is driven by~$Y_0^1$, as illustrated in Figure~\ref{fig2}.  \hfill$\square$

\def\stk#1#2#3{\begin{tikzpicture}[scale=0.35]
	\def\s{#1}
	\def\sm{#2}
	\coordinate (e1) at (10,0);
	\coordinate (e2) at (0,10);
	\coordinate (u1) at ($(e1)+(e2)$);
	\coordinate (u2) at ($\s*(e1) + \sm*(e2)$);
	\coordinate (u3) at ($\s*(e2) + \sm*(e1)$);
	\draw (0,0) -- (e1) -- ($(e1) + (e2)$) -- (e2) -- (0,0);
	\draw ($\s*(e2)$) -- ($\s*(e2) + (e1)$);
	\draw ($\s*(e1)$) -- ($\s*(e1) + (e2)$);
	\draw ($.5*\s*(u1)$) node {$\alpha$};
	\draw ($\s*(u1) + .5*\sm*(u1)$) node {$\beta$};
	\draw ($\s*(e2) + .5*(u2)$) node {$\delta$};
	\draw ($\s*(e1) + .5*(u3)$) node {$\delta$};
	\draw ($\s*(e1)$) node[yshift=-13pt] {#3};
\end{tikzpicture}
}
\def\stka#1#2#3{\begin{tikzpicture}[scale=0.35]
	\def\s{#1}
	\def\sm{#2}
	\coordinate (e1) at (10,0);
	\coordinate (e2) at (0,10);
 	\coordinate (u1) at ($(e1)+(e2)$);
	\coordinate (u2) at ($\s*(e1) + \sm*(e2)$);
	\coordinate (u3) at ($\s*(e2) + \sm*(e1)$);
	\draw (0,0) -- (e1) -- ($(e1) + (e2)$) -- (e2) -- (0,0);
	\draw ($\s*(e2)$) -- ($\s*(e2) + (e1)$);
	\draw ($\s*(e1)$) -- ($\s*(e1) + (e2)$);
	\draw ($.5*\s*(u1)$) node {$\alpha$};
	\draw ($\s*(u1) + .5*\sm*(u1)$) node {$\beta$};
	\draw ($\s*(e2) + .5*(u2)$) node {$\delta$};
	\draw ($\s*(e1) + .5*(u3)$) node {$\delta$};
	\draw ($\s*(e1)$) node[yshift=-13pt] {#3};
\end{tikzpicture}
}

\begin{figure}
\centering
\stka{.7}{.3}{ $Y^1_0(s)$}
\enskip
\caption{\label{fig2} 
Graphical representation~$h^1$, with~$Y^1_0(s)$ the fraction of Type~$0$ in the population at time~$s$.}
\end{figure}
\end{example}

\hspace{0.4cm}
In the above example we clearly see that, while the proportions of types are diffusive, the connection probabilities within and between types are constant. We next present an example where the connection probabilities diffuse as well.

\begin{example}\label{exam2}
Let~$r(u,v) = uv$, and define
\begin{equation*}
H^{m,n}\big(s;\tfrac{\ell}{m+1}\big) = Y^{m,n}_\ell(s), 
\quad H^m\big(s;\tfrac{\ell}{m+1}\big) = Y^{m}_\ell(s), \quad 0\leq\ell\leq m,
\end{equation*}
and, for~$x\neq \tfrac{\ell}{m+1}$, $0\leq\ell\leq m$, define~$H^{m,n}(s;x)$ and~$H^{m}(s;x)$ via linear interpolation. By Theorem~\ref{13}, as~$n\to\infty$ the graph process~$G^{m,n}$ converges to~$\tilde{h}^m$ given by its representative
\begin{equation*}
h^m (s;x,y) =  Y^m_j(s)  Y^{m}_\ell(s),
\end{equation*}
where~$j$ and~$\ell$ are such that
\beq{
\sum_{i=0}^{j-1} Y^{m}_i(s) \leq x < \sum_{i=0}^{j} Y^{m}_i(s),
\qquad
\sum_{i=0}^{\ell-1} Y^{m}_i(s) \leq y < \sum_{i=0}^{\ell} Y^{m}_i(s)
}
(for~$x=1$, respectively, $y=1$, we just set~$j=m$, respectively, $\ell=m$).
\hfill$\square$
\end{example}

\hspace{0.4cm}
If we let~$m\rightarrow \infty$ in the above example, then~$h^m (s)$ converges to the zero graphon. Consequently, we need to do adapt the fitness landscape in order to obtain a non-trivial limit. 

\begin{example}\label{exam3}
Let us consider $r$ as above, namely, 
\begin{equation*}
r(u,v) = uv.
\end{equation*} 
Consider a suitable \emph{mutual fitness function}, given by a continuous function $f \in \cW$, and let $c \in (0,1)$ be the \emph{fitness threshold}. For each $m \in \N$ and $0 \leq \ell \leq m$, define the mutual fitness partner sets as
\begin{equation*}
A_\ell^m = \left \{\tfrac{j}{m+1}\colon\, f\left(\tfrac{\ell}{m+1}, \tfrac{j}{m+1}\right) \geq c \right\}
\end{equation*}
and define the fitness landscape process as
\begin{equation*}
H^m\big(s;\tfrac{\ell}{m+1}\big) = \sum_{k \in A^m_\ell}Y^{m}_k(s), \quad 0\leq\ell\leq m.
\end{equation*}
Thus, individuals~$i$ and~$j$ are connected with probability
\begin{align*}
r\Big(H^m\big(s; \tfrac{i}{m+1}\big), H^m\big(s; \tfrac{j}{m+1}\big)\Big) 
= H^m\big(s; \tfrac{i}{m+1}\big)\, H^m\big(s;\tfrac{j}{m+1}\big).
\end{align*}
We note that, unlike in the earlier example, this probability depends on the rest of the population via the \emph{total fractions} of all individuals whose type has a \emph{sufficiently large mutual fitness} with respect to individuals of type ~$i$, respectively, type $j$. For the graphon process we may define
\begin{equation*}
H^m(s; u) = \int_{[0,1]}\I\big[f\big(\bar F^m(s;u),\bar F^m(s;v)\big) \geq c\big] Z^m(s;\dd v), \qquad u \in (0,1),
\end{equation*}
and define the corresponding graphon process~$\tilde{h}^m$ to be given by its representative
\begin{equation*}
h^m (s;x,y) =  H^m(s,x)\,H^m(s,y), \qquad x,y, \in (0,1).
\end{equation*}
It is easy to see that the $H^m$ satisfy \textbf{(H1)} with
\begin{equation*}
H(s;u) =\int_{[0,1]}\I\big[f\big(\bar F(s;u),\bar F(s;v)\big) \geq c\big] Z(s;\dd v).
\end{equation*}
By Theorem \ref{17}, we obtain that the graphon process ~$\tilde{h}^m$ converges as $m \rightarrow \infty$ to $\tilde{h}$, which is given by its representative
\begin{equation*}
h(s;x,y) =  H(s,x)H(s,y), \qquad x,y, \in (0,1).
\end{equation*}
Here, $r$ can be viewed as a type-connection matrix or type connection kernel that is controlled by the empirical type distribution of the entire population.

Note that $r$ need not take the product form, as illustrated by the fact that any $r \in \cW$ satisfying \textbf{(R1)} will suffice. Thus, the example's framework provides a rich class of non-trival diffusions in the space of graphons, well beyond the stochastic block model framework. \hfill$\square$
\end{example}


\subsection{Remarks}
\label{rm1}

The Moran model, the Wright-Fisher diffusion and the Fleming-Viot diffusion were used to create the graphon dynamics constructed in Theorems~\ref{13} and~\ref{17}. We mention three generalisations of these models where our results carry over.

\newcounter{myitemcounter}
\begin{enumerate}[label=(\alph*)]
\item 
In Section~\ref{intro}, we discussed the Moran model with~$n$ individuals, each carrying Type~$0$ or Type~$1$, and we identified the limiting process as the Wright-Fisher diffusion on~$[0,1]$ given by the SDE
\beq{
\dd Y(s) = \sqrt{Y(s)(1-Y(s))}\,\dd W(s)
}
with initial condition~$Y(0)$, where~$W$ is standard Brownian motion. This SDE has a unique strong solution, that is, there is a unique path~$s \mapsto Y(s)$ that is measurable with respect to the canonical filtration associated with the Brownian motion. The generator~$L$ of the Wright-Fisher diffusion is 
\begin{equation*}
(L\phi)(x) = x(1-x)\,\frac{\partial^2\phi(x)}{\partial x^2}, \qquad x \in [0,1],
\end{equation*}
for test functions~$\phi\colon\, [0,1] \to \R$ that are twice continuously differentiable. It is possible to generalise the Wright-Fisher diffusion by allowing for a \emph{state-dependent resampling rate}. Indeed, if individuals resample at rate~$s(x)$ when the fraction of individuals of type~$0$ is~$x$, then the generator becomes
\begin{equation*}
(L^s\phi)(x) = s(x)\, x(1-x)\frac{\partial^2\phi(x)}{\partial x^2}, \qquad x \in [0,1].
\end{equation*} 
In order for the SDE to be well-defined, some restrictions need to be imposed on the function~$s$, for instance, $s(x) > 0$ for~$x \in (0,1)$ and~$x \mapsto s(x)\,x(1-x)$ is Lipschitz on~$[0,1]$ (see \cite[Section 2]{Dawson1993}). An example is~$s(x)=x(1-x)$, which corresponds to Ohta-Kimura resampling at a rate that is proportional to the genetic diversity of the population. 

\item 
The Moran model with~$m+1$ types is the Markov process on the simplex~$\cS^m$ with generator
\begin{equation*}
(L^m\phi)(x) = \sum_{0 \leq k < \ell \leq m} x_k x_\ell\,\,\frac{\partial^2\phi(x)}{\partial x_k \partial x_\ell}, 
\qquad x \in \cS^m,
\end{equation*}
for test functions~$\phi\colon\,\cS^m \to \R$ that are twice continuously differentiable in each coordinate. Again, it is possible to modulate the resampling rate by a function~$s\colon\, \cS^m \to [0,\infty)$, but \emph{severe restrictions} need to be imposed on the behaviour of~$s$ near the boundary of~$\cS^m$ (for example, $s$ is close to~$1$; see \cite{Bass2008}).

\item
The Fleming-Viot diffusion with infinitely many types is the Markov process on the set~$\cP([0,1])$ with generator
\begin{equation*}
(L\phi)(x) = \int_{[0,1]^2}  
\big[x(\dd u)\delta_u(\dd v)- x(\dd u)x(\dd v)\big]\, \frac{\partial^2\phi(x)}{\partial x^2}[\delta_u,\delta_v], 
\end{equation*}
where
\begin{equation*}
\frac{\partial^2\phi(x)}{\partial x^2}[\delta_u,\delta_v]
= \frac{\partial}{\partial x}\left(\frac{\partial\phi(x)}{\partial x}[\delta_u]\right)[\delta_v],
\qquad u,v \in [0,1],
\end{equation*}
for test functions~$\phi\colon\, \cP([0,1]) \to \R$ of the form
\begin{equation*}
\phi(x) = \int_{[0,1]^n} x(\dd u_1) \times \cdots \times x(\dd u_n)\,\psi\big(u_1,\ldots,u_n\big),
\end{equation*}
where~$x  \in \cP([0,1])$, $n \in \N$, and~$\psi\colon\,[0,1]^n\to\R$ is continuous. Again, it is possible to modulate the resampling rate by a function~$s\colon\, \cP([0,1]) \to [0,\infty)$, but \emph{severe restrictions} need to be imposed on~$s$ in order to make sure that the diffusion is well-defined (for example, $s$ is close to~$1$; see \cite{Dawson1995}).
\end{enumerate}


\noindent The limiting graphons dynamics in Theorems~\ref{13} and~\ref{17} are Markov  processes on the space of graphons~$(\widetilde{\cal W},\delta_{\sub})$. Indeed, because
\begin{equation*}
\tilde h(s) = G(Y(s)), \qquad s \geq 0,
\end{equation*}
for some invertible map~$G\colon\, \cS^m \to \widetilde{\cal W}$ or~$G\colon\, \cP([0,1]) \to \widetilde{\cal W}$, the generator~$\tilde L$ of the graphon dynamics can be \emph{formally} computed as
\begin{align*}
(\tilde L\phi)\bclr{\tilde h(0)} 
&= \lim_{t \downarrow 0} \frac{\EE\bclc{\phi\bclr{\tilde h(s)} - \phi\bclr{\tilde h(0)} \given \tilde h(0)}}{s}\\
&= \lim_{t \downarrow 0} \frac{\EE\bclc{(\phi \circ G)(Y(s)) - (\phi \circ G) (Y(0)) \given Y(0)}}{y}\\
&= L(\phi \circ G)(Y(0)) = L(\phi\circ G)\bclr{G^{-1}\bclr{\tilde h(0)}},
\end{align*}
for appropriate test functions~$\phi\colon\, \widetilde{\cal W} \to \R$. Hence we have the representation
\begin{equation*}
\tilde L \phi = L(\phi \circ G) \circ G^{-1}.
\end{equation*}
An explicit form of the generator may be possible for specific choices of~$G$.


\section{Proof of main results}
\label{18}

\subsection{Proof of Theorem~\ref{thm1}} 

\begin{proof}
``$(i)\Longrightarrow(ii)$''. Since~$t_F$ is a continuous function from~$(\widetilde{\cal W},\dsub)$ to~$[0,1]$, and thus continuous from~$D([0,\infty),\widetilde{\cal W})$ to~$D([0,\infty),[0,1])$ by \cite[Problem~13, p.\,151]{Ethier1986}, it is clear that~$(i)$ implies~$(ii)$ by the continuous mapping theorem. 

\smallskip\noindent
``$(ii)\Longrightarrow(iii)$''. Now, observe that~$(ii)$ for~$k=1$ implies tightness of~$(t_F(\tilde{h}_n))_{n\in\N}$, which is the first part of~$(iii)$. For points of continuity of~$(t_{F_1}(\tilde{h}),\dots,t_{F_k}(\tilde{h}))$, $(ii)$ implies that the finite dimensional distributions converge which implies the second part of~$(iii)$. 

\smallskip\noindent 
``$(iii)\Longrightarrow(i)$''. Note first that functions of the form
\beqn{\label{19}
f=\sum_{i=1}^m a_i t_{F_i}, \qquad \text{$m\in\N$, $a_1,\dots,a_m\in\R$, $F_1,\dots, F_n\in \cF$,}
}
form an algebra because, for any two graphs~$F_1$ and~$F_2$, 
\beq{
t_{F_1}(h) t_{F_2}(h) = t_{F_1\uplus F_2}(h), 
\qquad\text{$h\in\wt\cW$, $F_1,F_2\in\cF$},
}
where~${F_1\uplus F_2}$ denotes the disjoint union of the two graphs (that is, no merging of vertices), which is again in~$\cF$. By \cite[Theorem~2.2]{Diao2015}, this set of functions is dense in the space of continuous functions on~$\widetilde{\cal W}$ with respect to the topology of uniform convergence (the cited theorem is with respect to a metric that is equivalent to~$d_{\sub}$ with respect to convergence). Since, by assumption, all~$(t_{F_i}(\tilde{h}_n))_{n\in\N}$ are relatively compact, so are finite collections of them, and therefore so are~$(f(\tilde{h}_n))_{n\in\N}$ for all~$f$ of the form~\eq{19}, which implies relative compactness of~$(\tilde{h}_n)_{n\in\N}$ by \cite[Theorem 9.1, p.\,142]{Ethier1986} (the \emph{compact containment condition} in that theorem trivially holds because~$\widetilde{\cal W}$ itself is compact). 

In order to establish convergence of the finite-dimensional distributions of~$\tilde{h}_n$ to~$\tilde{h}$, we next show that the family of functions~$\{t_F\colon\, F\in\cF\}$ \emph{strongly separates points} in~$(\widetilde{\cal W},\dsub)$, that is, we show that for each~$h\in\wt\cW$ and each~$\delta>0$ there exist~$F_1,\dots,F_k\in \cF$ such that
\beq{
\inf_{h'\colon\,\dsub(h,h')\geq\delta}\,\, \max_{1\leq i\leq k}\abs{t_{F_i}(h)-t_{F_i}(h')} > 0
}
(c.f.\ \cite[p.\,113]{Ethier1986}). So, let~$h\in\wt\cW$ and~$\delta>0$ be given. Let~$(F_i)_{i\in\N}$ be the enumeration of~$\cF$ giving rise to~$\dsub$, and choose~$m$ such that 
\beq{
\sum_{i>m} 2^{-i} \leq \delta/2.
}
If~$h'$ is such that~$d_\sub(h,h')\geq \delta$,then we must have
\beqn{\label{20}
\sum_{i=1}^m 2^{-i}\babs{t_{F_i}(h)-f_{F_i}(h')}\geq \frac{\delta}{2}.
}
If all differences~$\babs{t_{F_i}(h)-t_{F_i}(h')}$ were strictly smaller than~$\delta/2m$, then we would have
\beq{
\sum_{i=1}^m 2^{-i}\babs{t_{F_i}(h)-f_{F_i}(h')}\leq
\sum_{i=1}^m \babs{t_{F_i}(h)-f_{F_i}(h')}<\delta/2,
}
which contradicts~\eq{20}, and so at least one of the differences must be larger than~$\delta/(2m)$. Hence 
\beq{
\inf_{h'\colon\,d_\sub(h,h') \geq \delta} \,\, \max_{1\leq i\leq m}
\babs{t_{F_i}(h)-f_{F_i}(h')} \geq \delta/2m,
}
which establishes that~$\{t_F\colon\, F\in\cF\}$ strongly separate points, and therefore, by \cite[Theorem~4.5, p.\,113]{Ethier1986}, is convergence determining. By \cite[Proposition~4.6(b), p.\,115]{Ethier1986}, functions of the form~$t_{F_1}\cdots t_{F_k}$ are convergence determining on the product space~$\wt\cW^k$ equipped with the canonical metric induced by~$\dsub$, and so convergence of the finite-dimensional distributions follows from~\eq{7}. This establishes~$(i)$.
\end{proof}

\subsection{Proof of Theorem~\ref{13}}

\begin{proof}
Using the Skorokhod embedding, we may assume that the processes~$X^{m,n}$, $Y^{m,n}$, $H^{m,n}$, $Y^{m}$ and~$H^{m}$ are constructed in such a way that 
\beqn{
\bclr{Y^{m,n},H^{m,n}} \asconv (Y^m,H^m)\quad(n\toinf) \qquad\text{$\PP$-almost surely.}\label{21}
}
We write~$\PP'$ and~$\EE'$ to denote conditional probability and conditional expectation given $X^{m,n}$,  $Y^{m,n}$, $H^{m,n}$, $Y^{m}$ and~$H^{m}$. We will establish~$(iii)$ of Theorem~\ref{thm1}  to prove our result. We do this in two steps.

\smallskip\noindent
{\bf Step 1: Proof of (\ref{7}).} We first prove a univariate concentration argument of the subgraph densities around their mean. Fix~$s \geq 0$ and a finite graph~$F$. Recall McDiarmid's concentration inequality from \cite{McDiarmid1989}. If~$W=(W_1,\dots,W_N)$ is a collection of~$N$ independent random variables and~$f$ is a function of~$N$ variables such that an arbitrary change of the~$i$-th variable changes the value of~$f$ by at most~$c_i$, then 
\begin{equation}\label{23}
\PP\big[|f(W) - \EE\{f(W)\} | \geq \eps\big] 
\leq 2\exp\bbbclr{\textstyle{-\displaystyle\frac{2\eps^2}{\sum_{i=1}^N c_i^2}}}.
\end{equation} 
Since~$F$ has~$k$ vertices and~$G^{m,n}(s)$ has~$n$ vertices, it follows that~$t^{\inj}_{F}( G^{m,n}(s))$ changes by at most~${n-2\choose k-2}/{n\choose k}={k \choose 2}/{n \choose 2}$ when one edge is changed. Applying~\eq{23} with~$N = {n \choose 2}$ we have, for any~$\eps>0$,
\begin{equation}\label{24}
\PP'\Big[\babs{t^{\inj}_{F}( G^{m,n}(s))-\EE'\{ t^{\inj}_{F}( G^{m,n}(s))\}} > \eps\Big]
\leq 2\exp\bbbclr{-2\eps^2 {n \choose 2}{k \choose 2}^{-2}}.
\end{equation}
The  mean of~$t^{\inj}_{F}( G^{m,n}(s))$ with respect to the connection probabilities can be expressed as
\beqsn{\label{25}
& \EE'\{t^{\inj}_{F}( G^{m,n}(s))\}\\
&\quad = \sum_{\ell_1,\dots,\ell_k=0}^m p_{X^{m,n}_0(ns),\dots,X^{m,n}_m(ns)}(\ell_1,\dots,\ell_k)\\
&\kern10em\times\prod_{\{i,j\}\in E(F)} 
r\Big(H^{m,n}\big(s;\tfrac{\ell_i}{m+1}\big), H^{m,n}\big(s;\tfrac{\ell_j}{m+1}\big)\Big),
}
where~$p_{K_0,\dots,K_m}(\ell_1,\dots,\ell_k)$ is the probability that an ordered sample of~$k$ balls drawn without replacement from an urn with~$K_i$ balls of colour~$i$, where~$0\leq i\leq m$, is such that the~$i$-th ball has colour~$\ell_i$, $0\leq i\leq m$. We note that there is a constant~$C>0$ (which may depend on~$m$ and~$k$ but not on~$n$) such that
\beqn{\label{26}
\bbbabs{ p_{X^{m,n}_0(ns),\dots,X^{m,n}_m(ns)}(\ell_1,\dots,\ell_k)
- \prod_{i=1}^k Y^{m,n}_{\ell_i}(s) } \leq \frac{C}{n}
\qquad\text{for all  $n \in \N$.}
}
Using continuity of~$r$,~\eq{21} and~\eq{26}, we conclude that, $\PP$-almost surely,
\beqsn{\label{27}
& \lim_{n\toinf}\EE'\{t^{\inj}_{F}( G^{m,n}(s))\} \\
& \qquad = \sum_{\ell_1,\dots,\ell_k=0}^m\, \prod_{i=1}^k Y^{m}_{\ell_i}(s) 
\prod_{\{i,j\}\in E(F)} r\Big(H^{m}\big(s; \tfrac{\ell_i}{m+1}\big), H^{m}\big(s;\tfrac{\ell_j}{m+1}\big)\Big)\\
&\qquad = t_F(\tilde{h}^m(s))
}
whenever~$s$ is a point at which~$(Y^m,H^m)$ is continuous almost surely; here we also need the fact that~$H^{m,n}(s)\in C([0,1],[0,1])$, so that, for~$x$ fixed, $H^{m,n}(s;x)$ is a continuous function of~$H^{m,n}(s)$ and converges to~$H^{m}(s;x)$. Using Borel-Cantelli with~\eq{24}, followed by~\eq{27}, we obtain
\beq{
t^{\inj}_{F}( G^{m,n}(s)) \asconv  t_{F}( \tilde{h}^m(s))\quad(n\toinf)\qquad \text{$\PP$-almost surely.}
}
Let~$d \geq 1$, let~$0 < s_1 < s_2 < \ldots <s_d$ be points at which~$(Y^m,H^m)$ is continuous almost surely, and let~$F_1, F_2, \ldots, F_d$ be finite sub-graphs. Then, by the above, we have
\beq{
\prod_{i=1}^d t^\inj_{F_i}(  G^{m,n}(s_i)) \asconv \prod_{i=1}^d  t_{F_i}( \tilde{h}^m(s_i))\quad(n\toinf)\qquad 
\text{$\PP$-almost surely.}
}
Using the bounded convergence theorem, we conclude that
\beq{
\lim_{n\toinf}\EE\bbbclc{\prod_{i=1}^d t^\inj_{F_i}( G^{m,n}(s_i))} = \EE\bbbclc{\prod_{i=1}^d  t_{F_i}( \tilde{h}^m(s_i))},
}
which establishes~\eq{7}.

\smallskip\noindent
\textbf{Step 2: Tightness.} Fix a finite graph~$F$ on~$k$ vertices. We need to show that the family of processes~$\bclr{t^\inj_F( G^{m,n})}_{n\in\N}$ is tight. We show this via tightness of some intermediate processes.  First, for each~$n\geq 1$, let
\beq{
W^{m,n}(s) \Def \sum_{\ell_1,\dots,\ell_k=0}^m\, \prod_{i=1}^k Y^{m,n}_{\ell_i}(s) 
\prod_{\{i,j\}\in E(F)} r\Big(H^{m,n}\big(s;\tfrac{\ell_i}{m+1}\big), H^{m,n}(s;\tfrac{\ell_j}{m+1}\big)\Big), \quad s\geq 0.
}
Note that~$W^{m,n}$ is a continuous function of~$Y^{m,n}$ and~$H^{m,n}$. Since~$\clr{Y^{m,n},H^{m,n}}_{n\geq 1}$ converges weakly and is therefore tight, and since compact sets remain compact under continuous mappings, it follows that~$\clr{W^{m,n}}_{n\geq 1}$ is tight. Second, for each~$n\geq 1$, let
\beq{
V^{m,n}(s) \Def \IE'\bclc{t^\inj_F( G^{m,n}(s))}, \quad s\geq 0.
}
Recalling~\eq{25} and applying~\eq{26}, it follows that
\beq{
  \sup_{s\geq 0}\babs{V^{m,n}(s)-W^{m,n}(s)}\leq \frac{C}{n},
}
and so, by \cite[Inequality~(6.3), p.\,122]{Ethier1986} and \cite[Corollary~7.4, p.\,129]{Ethier1986} and tightness of~$(W^{m,n})_{n\geq 1}$, it follows that~$(V^{m,n})_{n\geq 1}$ is also tight. 

Third, observe that the sequence of events when vertices resample their type follows a Poisson point process with rate~$n$ (since there are~$n$ vertices and each vertex resamples at rate~$1$), so that the number of resampling events in the time interval~$[0,Tn]$, where~$T>0$ is fixed, follows a Poisson distribution with mean~$n\times Tn$. Denote by~$A_{n,m}$ the event that the process~$Y^{m,n}$ saw no more than~$n^3$ such resampling events in the interval~$[0,T]$, and so, letting~$Z\sim\Po(n^2T)$, we have
\beqn{\label{27a}
  \IP[A_{m,n}^c] = \IP[Z>n^3]\leq\frac{\IE Z^2}{n^6} \leq \frac{2T^2}{n^2}, \qquad n^2\geq \frac{1}{T}.
}
On the event~$A_{n,m}$, we now apply a union bound over all resampling events and then apply~\eq{24}, which yields
\beq{
  \PP'\bbbcls{\,\sup_{0\leq s\leq T}\babs{t^{\inj}_{F}( G^{m,n}(s))-V^{m,n}(s)} > \eps\given A_{n,m}}
  \leq 2n^3\exp\bbbclr{-2\eps^2 {n \choose 2}{k \choose 2}^{-2}}.
}
Choosing 
\beq{
  \eps^2 = 2\log(n){k\choose2}^2{n\choose 2}^{-1},
}
we conclude that
\beqn{\label{27b}
  \PP'\bbbcls{\,\sup_{0\leq s\leq T}\babs{t^{\inj}_{F}( G^{m,n}(s))-V^{m,n}(s)} > \frac{k\log n}{n}\given A_{n,m}}
    \leq \frac{2}{n}.
}
Using~\eq{27a} and~\eq{27b} and tightness of~$(V^{m,n})_{n\geq1}$, it is now straightforward to check Condition (b) of \cite[Corollary~7.4, p.\,129]{Ethier1986} by also using \cite[Inequality~(6.3), p.\,122]{Ethier1986}, and it follows that~$\bclr{t^{\inj}_{F}( G^{m,n})}_{n\geq 1}$ is tight.
\end{proof}


\subsection{Proof of Theorem~\ref{17}}

\begin{proof}
Using the Skorokhod embedding, we may assume that the processes~$Z^{m}$, $H^{m}$, $Z$ and~$Z^{m}$ are constructed in such a way that 
\beqn{
(Z^{m},H^m) \asconv (Z,H)\quad(m\toinf) \qquad\text{$\PP$-almost surely. } \label{35}
}
We first establish the following fact. As before, denote by~$\cP([0,1])$ the set of probabilty measures on~$[0,1]$ endowed with the topology of weak convergence, and let~$\lambda:\cP[0,1]\times C([0,1],[0,1])\to\RR$ be defined as
\beq{
  \lambda(\mu,f) = \int_{[0,1]^k} \prod_{\{i,j\}\in E(F)} r(f(x_i),f(x_j))\, \mu(\dd x_1)\cdots\mu(\dd x_k).
}
Then~$\lambda$ is a continuous function. To see this, take a sequence~$(\mu_n,f_n)$ that converges to~$(\mu,f)$. Write
\beqs{
  \abs{\lambda(\mu_n,f_n)-\lambda(\mu,f)}
  \leq \abs{\lambda(\mu_n,f_n)-\lambda(\mu_n,f)}+
  \abs{\lambda(\mu_n,f)-\lambda(\mu,f)}~\eqqcolon r_{1,n} + r_{2,n}.
} 
In order to show that~$r_{1,n}$ converges to zero, it suffices to show that the quantity~$r'_{1,n}\Def\int r(f_n(x),f_n(y))\pi(dx)\pi(dy)$ converges to zero uniformly in~$\pi\in\cP([0,1])$ and then use a telescoping sum and the fact that~$r$ is bounded by~$1$. But since~$r$ and~$f$ are defined on compact sets and are therefore uniformly continuous, it is easy to establish that~$r'_{1,n}$ converges to zero uniformly in~$\pi$. The fact that~$r_{2,n}$ converges to zero follows from the fact that~$\mu_n$ converges weakly to~$\mu$ and that therefore the~$k$-fold convolution converges weakly, too.

We now proceed to establish~$(iii)$ of Theorem~\ref{thm1}. 

\medskip\noindent 
{\bf Step 1: Proof of (\ref{7}).}
For every simple finite graph  $F$ on~$k$ vertices, ~\eq{2} yields the representations
\beq{
t_{F}(\tilde h^m(s))  
=\lambda(Z^m,H^m(s)),\qquad 
t_{F}(\tilde h (s))  
= \lambda(Z,H(s)).
}
Let~$s$ be a point at which~$(Z,H)$ is continuous almost surely. By continuity of~$\lambda$, we conclude that 
\begin{equation*}
\lim_{m\to\infty} t_{F}(\tilde h^m(s)) = t_{F}(\tilde h(s))\quad\text{almost surely.}
\end{equation*}
For~$d\geq 1$ and successive times~$0 \leq s_1< \ldots <s_d < \infty$ at which~$(Z,H)$ is continuous almost surely,~\eq{7} can now be established in the same way as in the previous proof.

\medskip\noindent 
{\bf Step 2: Tightness.} Fix a finite graph~$F$ on~$k$ vertices; we need to show that the family of processes~$\bclr{t_F(\tilde h^{m})}_{m\geq 1}$ is tight. Noting that the function that maps the process~$(Z^m,H^m)$ to~$\lambda\circ(Z^m,H^m)$ is continuous since~$\lambda$ is continuous (c.f.\ \cite[Problem~13, p.\,151]{Ethier1986}), tightness is immediate because it is preserved under continuous mappings.
\end{proof}

  
\section*{Acknowledgements}

SA was supported in part through an ISF-UGC grant and a CPDA from the Indian Statistical Institute. FdH was supported through NWO Gravitation Grant NETWORKS-024.002.003. AR was supported through Singapore Ministry of Education Academic Research Fund Tier 2 grant MOE2018-T2-2-076. Part of this work was done while SA was visiting the Indian Institute of Science, Bangalore, and SA and FdH were visiting the International Center for Theoretical Sciences, Bangalore. The hospitality that was provided at these institutes is greatly appreciated.


\setlength{\bibsep}{0.5ex}
\def\bibfont{\small}


\end{document}